\documentclass{article}
\usepackage{amssymb}
\usepackage{amsmath}
\usepackage{amsthm}
\usepackage{verbatim}
\usepackage{color}
\usepackage{url}
\usepackage{fancyhdr}

\newtheorem{thm}{Theorem}

\newtheorem{lem}{Lemma}
\newtheorem{cor}{Corollary}

\newcommand{\sabs}[1]{\left|#1\right|}
\newcommand{\sparen}[1]{\left(#1\right)}
\newcommand{\norm}[1]{\sabs{\sabs{#1}}}

\numberwithin{equation}{section}

\title{Stable determination of X-ray transforms of time dependent potentials from partial boundary data}
\author{Alden Waters,\\CNRS Ecole Normale Superieure, Paris\\ alden.waters@gmail.com}
\begin{document}

\maketitle

\begin{abstract}
We consider compact smooth Riemmanian manifolds with boundary of dimension greater than or equal to two. For the initial-boundary value problem for the wave equation with a lower order term $q(t,x)$, we can recover the X-ray transform of time dependent potentials $q(t,x)$ from the dynamical Dirichlet-to-Neumann map in a stable way. We derive conditional H\"older stability estimates for the X-ray transform of $q(t,x)$. The essential technique involved is the Gaussian beam Ansatz, and the proofs are done with the minimal assumptions on the geometry for the Ansatz to be well-defined.  
\end{abstract}



\section{Introduction to notational conventions}
We consider a Riemannian manifold $\mathcal{M}$ equipped with a metric $g$. We use the standard Einstein summation convention for the rest of this paper.  We let $\Delta_g$ denote the Laplace-Beltrami operator, which we write as
\begin{align}
\Delta_g=\frac{1}{\sqrt{\det g(x)}}\frac{\partial}{\partial x^k}\sparen{g^{ki}(x)\sqrt{\det g(x)}\frac{\partial}{\partial x^i}}
\end{align}
in local coordinates with $g(x)=(g_{ik}(x))$, and $(g^{ik}(x))=(g_{ki}(x))^{-1}$. We consider manifolds, $\mathcal{M}$, which are smooth ($C^{\infty})$. The local coordinates we abbreviate as $(x_1, . . ,x_n)$. We also assume the manifolds have a boundary.  

For this paper, we use many of the notational conventions in \cite{DDSF}. We let $(\frac{\partial}{\partial x^1}, . . ,\frac{\partial}{\partial x^n})$ denote the tangent vector fields so that the corresponding the inner product and norm on the tangent space $T_x\mathcal{M}$ are denoted by 
\begin{align*}
g(X,Y)=\langle X,Y\rangle_g=g_{jk}\alpha_j\beta_k
\end{align*}
\begin{align*}
|X|_g=\langle X,X \rangle_g^{\frac{1}{2}}, \qquad X=\alpha_i\frac{\partial}{\partial x_i} \qquad Y=\beta_i\frac{\partial}{\partial x_i}.
\end{align*}
Whenever $f$ is a $C^1$ function on $\mathcal{M}$, the gradient of $f$ is defined as the vector field $\nabla_g f$ so that $\forall X$ on $\mathcal{M}$ we have
\begin{align*}
X(f)=\langle \nabla_g f,X\rangle_g. 
\end{align*}
In local coordinates, we can write
\begin{align*}
\nabla_g f=g^{ij}\frac{\partial f}{\partial x_i}\frac{\partial}{\partial x_j}.
\end{align*}
The metric tensor induces a Riemannian volume form which as in \cite{DDSF} we denote by, 
\begin{align*}
d_gV=(\det g)^{\frac{1}{2}}dx_1\wedge . . . \wedge dx_n.
\end{align*}
The space $L^2(\mathcal{M})$ is the completion of $C^{\infty}(\mathcal{M})$ with respect to the inner product
\begin{align*}
\langle f_1,f_2\rangle=\int\limits_{\mathcal{M}}f_1(x)f_2(x)\,d_gV \qquad f_1,f_2\in \mathcal{C}^{\infty}(\mathcal{M}).
\end{align*}
We can define the Sobolev spaces for the manifolds analogously to the Euclidean Sobolev norms, so that 
\begin{align*}
\norm{f}^2_{H^1(\mathcal{M})}=\norm{f}_{L^2(\mathcal{M})}^2+\norm{\nabla f}_{L^2(\mathcal{M})}^2.
\end{align*}
With these definitions in mind, we consider the solutions $u(t,x)$ to the initial-boundary value problem
\begin{align}\label{hyperbolic1}
&(\Box_g+q(t,x))u(t,x)=F(t,x) \quad\,\,\, \mathrm{on} \qquad (0,T)\times \mathcal{M}
\\&u(t,x)|_{t=0}=\partial_t u(t,x)|_{t=0}=0 \qquad \mathrm{in} \qquad \mathcal{M} \nonumber
\\&u(t,x)=f(t,x) \qquad \qquad \quad \mathrm{on} \qquad (0,T) \times \partial\mathcal{M} \nonumber
\end{align} 
where 
\begin{align*}
\Box_g=\partial_t^2-\Delta_g.
\end{align*}
We know that this problem is well-posed, since we have the following existence and uniqueness result, see \cite{lions} or \cite{DDSF}
\begin{lem}\label{wellposed}
Assuming $f(t,x)\in H_0^1([0,T]\times \partial M)$, $F(t,x)\in L^1([0,T];L^2(\mathcal{M}))$, and $q(t,x)\in C^{0}([0,T]\times\mathcal{M})$ 
 then there exists a unique solution $u(t,x)$ to (\ref{hyperbolic1}), such that
 \begin{align*}
 u(t,x)\in C([0,T];H^1_0(\mathcal{M}))\cap C^1([0,T]; L^2(\mathcal{M}))
 \end{align*}
with norm bounds
\begin{align}\label{normbound}
&\sup_{t\in [0,T]}\sparen{\norm{u(t,x)}_{H^1(\mathcal{M})}+\norm{\partial_tu(t,x)}_{L^2({\mathcal{M})}}}
 \leq \\&
  C\sparen{\norm{F(t,x)}_{L^1([0,T];L^2(\mathcal{M}))} +\norm{f(t,x)}_{H^1_0([0,T]\times \partial M)}}
  \nonumber
\end{align}
where the constant $C$ is independent of $F(t,x)$ and $f(t,x)$ but depends on $\norm{q(t,x)}_{C^{0}([0,T]\times\mathcal{M})}$. 
\end{lem}
Since the solution is well-posed, we can introduce the problem of recovering the potential from the dynamical Dirichlet-to-Neumann map. We let $\nu=\nu(x)$ be the outer unit normal to $\partial\mathcal{M}$ at $x$ in $\partial\mathcal{M}$ which we normalize so we have
\begin{align*}
g^{kl}(x)\nu_k(x)\nu_l(x)=1.
\end{align*}
The dynamical Dirichlet-to-Neumann map, $\Lambda_{g,q}$, is defined by 
\begin{align*}
\Lambda_{g,q}f(t,x)= \nu_k(x)g^{kl}(x)\frac{\partial u}{\partial x_l}(t,x)|_{\sparen{(0,T)\times \partial \mathcal{M}}}.
\end{align*}
The natural norm on the Dirichlet-to-Neumann map is the operator norm from 
\begin{align*}
H_0^1(\partial\mathcal{M}\times(0,T))\rightarrow L^2(\partial\mathcal{M}\times(0,T))
\end{align*}
which we denote by 
\begin{align*}
\norm{\cdot}_{H_0^1\rightarrow L^2}.
\end{align*}
It follows from Lemma \ref{wellposed} that $\Lambda_{g,q}$ is bounded as a linear operator whenever $f(t,x)$ is in $H^1_0([0,T]\times\partial\mathcal{M})$, $F(t,x)\in L^1([0,T];L^2(\mathcal{M}))$, and $q(t,x)\in C^{0}([0,T]\times\mathcal{M})$. We know that $H_0^1((0,T)\times \partial\mathcal{M})$ is the completion of the space of $C_c^{\infty}((0,T)\times \partial\mathcal{M})$ functions with respect to the appropriate inner product. We mention this because in Lemma \ref{error} it is essential that the input of the Dirichlet-to-Neumann maps have compact support in the boundary cylinder. 

The question that this paper seeks to address is if we know the Dirichlet-to-Neumann maps of two different potentials, $q_1(t,x)$ and $q_2(t,x)$ what information about the X-ray transform of their difference can we gain?  Recent work on stability estimates by Bellasoued and Dos Santos Ferreira, \cite{DDSF} builds on the work by Kenig et al. \cite{dks}, which considers only elliptic Schr\"odinger operators. The work here will be largely inspired by \cite{DDSF} and Section 7 of Kenig and Salo \cite{ks}. We extend their constructions by considering a more general geometry than in \cite{DDSF} for the hyperbolic problem, and we consider the case of time dependent potentials. 

For this paper we assume that the manifold $\mathcal{M}$ is an arbitrary smooth compact Riemannian manifold with smooth boundary of dimension $n\geq 2$. The goal of this paper is to show that we can recover integrals over geodesics of potentials $q(t,x)$- the X-ray transform of potentials on $\mathcal{M}$, in a stable way from the dynamical Dirichlet-to-Neumann map. We allow work in integral geometry to tell us which class for which class of manifolds and admissible potentials we have stability and uniqueness results. We direct the reader to the preprint Uhlmann and Vasy \cite{vasy} for recent injectivity and stability results on the X-ray transform. Future works by integral geometers regarding X-ray transforms will produce better stability results of the potentials. 

In Eskin \cite{eskin1}, \cite{eskin2}, \cite{eskin3}, uses the boundary control method first introduced by Belishev \cite{belishev}, and Belishev and Kurylev \cite{bk}. Eskin's methods in \cite{eskin3} require the potential to be analytic in time.  For a survey on the literature of the boundary control method, and explanation of the techniques, one should see the monograph by Lassas et. al, \cite{LKK}. For some stability and uniqueness results for other partial differential equations with time-independent coefficients, see for example \cite{bj1}, \cite{bj2}, \cite{street}, \cite{ammari}, \cite{Klibanov} and \cite{Bukhgeim}.  Here the author has chosen to focus on references which are related to work on the X-ray transform for uniqueness and stability estimates for the hyperbolic problem. 

Using Green's theorem as in Alessandrini and Sylvester, \cite{as} and Sylvester and Uhlmann, \cite{su} and complex geometric optics to produce the X-ray transform, we derive stability results for the X-ray transform of potentials $q(t,x)$. The author uses only the minimal amount of assumptions on the geometry for the Gaussian beam Ansatz to be well-defined. In particular, the metric must be at least three times differentiable, an assumption which was also used in \cite{bao}. It seems likely that using the same techniques developed by the author in \cite{waters}, wave equations with  $C^{1,1}(t,x)$ coefficients could also be examined.

The study of the initial boundary value problem (\ref{hyperbolic1}) has a long history, and these results are formulated building on the results of others. For references in this direction, we direct the reader to the articles by Isakov \cite{isakov1}, \cite{isakov2} and Sun \cite{sun1}, and Isakov and Sun \cite{sun2}. Using X-ray transform methods, the first uniqueness result for time dependent potentials for wave equations was established by Stefanov in \cite{stefanov1} using the scattering relation when the geometry is Euclidean. Later, Sj\"ostrand and Ramm in $\mathbb{R}^n$ in \cite{sjo} established uniqueness results for time dependent potentials using the standard real-phase geometric optics Ansatz.

In \cite{DDSF} an extension of the techniques in \cite{sjo} and those of \cite{dks} to produce stability estimates for $q(x)$ and conformal factors $c(x)$ when considering simple manifolds and time independent lower order coefficients. Part of the results by Stefanov and Uhlmann in \cite{stefanovstable}, show that its possible to recover the conformal factor in a stable way when the manifold is simple. The first uniqueness results in this direction, using X-ray transform methods, are given by Rakesh \cite{rakesh} and Rakesh and Symes \cite{rakeshsymes}. Montalto in \cite{M} recovers the metric, conformal factor and lower order terms simultaneously for from the Dirichlet-to-Neumann (DN) map. He is able to conclude H\"older stability estimates for these coefficients from the DN-map for simple manifolds. Naturally in some places the construction will be similar to \cite{M}.

In \cite{DDSF} the assumption the manifold is simple is essential and used often in $\cite{stefanovstable}$ and $\cite{M}$. However, stability estimates for the X-ray transform of time dependent potentials and general geometric settings have been previously unobserved for a single measurement from the boundary at the same time. The generality of the geometry is suggested by the much earlier treatment of quasimodes by Ralston \cite{ralston77}. The author chose to follows some arguments in \cite{DDSF} closely because the goal was to replace the cutoff function in Remark 3, in \cite{DDSF}. This is done by using a good kernel argument from Stein and Shakarchi \cite{steinreal}, but was inspired by the use of the Gaussian beam Ansatz in \cite{ks} in the related elliptic case and also by \cite{bao}. However, in \cite{ks} it is essential that all of the potentials are time independent because of course, the operator is elliptic. Here we consider the case for time dependent potentials so the analysis is different. The Appendix attempts to relate the work to \cite{ks} and \cite{msd}. 

As aforementioned, Stefanov and Uhlmann in \cite{stefanovstable} proved uniqueness and stability results for simple metrics using the scattering relation. One of the main goals of their paper is to examine the boundary distance rigidity problem. The boundary rigidity problem is also examined in \cite{crokes} and \cite{croke}. If we had stronger assumptions on the stability of the X-ray transform for a more general geometry as in \cite{stefanovstability} then we could derive stability results for conformal factors as in \cite{DDSF}, \cite{dks} from the techniques in this paper. The goal here is different than \cite{stefanovstable}   because the focus is on the lower order terms. We derive H\"older stability estimates for the X-ray transform of $q(t,x)$ from the DN map by constructing solutions to the wave equation and using Green's theorem. 

\vskip 10pt
\textbf{Acknowledgements:} The author would like to thank James Ralston for his ongoing encouragement to pursue mathematics. She would like to thank Mikko Salo for his suggestion to investigate this problem in the context of his work.  The author was supported by a postdoctoral fellowship at the University of Jyv\"askyl\"a and an AXA Foundation research grant at the Institut Mittag-Leffler. The author is currently supported by CNRS.

\section{Statement of the Main Theorems}

We now introduce the definition of the X-ray transform on the manifolds we will be using. During the course of this paper we will use many of the notational conventions in \cite{DDSF}. For $x\in \mathcal{M}$ and $\omega\in T_x\mathcal{M}$ we let $\gamma_{x,\omega}$ denote the unique geodesic with initial conditions
\begin{align*}
\gamma_{x,\omega}(0)=x \qquad \dot{\gamma}_{x,\omega}(0)=\omega.
\end{align*}
We let 
\begin{align}
\mathcal{SM}=\{(x,\omega)\in T\mathcal{M};\quad |\omega|_g=1\}
\end{align}
denote the sphere bundle of $\mathcal{M}$. We let the submanifold of inner vectors of $\mathcal{SM}$ be denoted by
\begin{align*}
\partial_{+}\mathcal{SM}=\{(x,\omega)\in \mathcal{SM},\quad x\in \partial\mathcal{M},\quad +\langle\omega,\nu(x)\rangle<0\}.
\end{align*}
Let $\tau(x,\omega)$ be the length of the geodesic segment with initial conditions $(x,\omega)\in\partial_+\mathcal{SM}$. We consider the inward pointing vectors of $\mathcal{SM}$ only. We assume our manifold with boundary, $\mathcal{M}$ has the property that:
\begin{itemize}
\item There is a $T$ such that for all $(x,\omega)\in\partial_+\mathcal{SM}$ there is a $\tau(x,\omega)\leq T$ such that $\gamma_{x,\omega}(t)$ is in the interior of $\mathcal{M}$ for $0<t<\tau(x,\omega)$, and intersects the boundary $\partial\mathcal{M}$ transversally when $t=\tau(x,\omega)$.
\end{itemize}
This hypothesis is the weakest assumption on the geometry for the Gaussian beam Ansatz to be well-defined in this setting. This was also observed in \cite{bao}. 

The definition of the X-ray transform on time dependent functions $f(t,x)$ we are using is as follows:
\begin{align}\label{Xray}
I_{x,\omega} f=\int\limits_0^{\tau(x,\omega)}f(s,\gamma_{x,\omega}(s))\,ds.
\end{align}
The right hand side of (\ref{Xray}) is a smooth function on the space $\partial_{+}\mathcal{SM}$ because the integration bound $\tau(x,\omega)$ is a smooth function on $\partial_+\mathcal{SM}$.
 For convex non-trapping manifolds, the ray transform on time independent functions can be extended as a bounded operator 
\begin{align*}
I: H^k(\mathcal{M}) \rightarrow H^k(\partial_{+}\mathcal{SM})
\end{align*}
for all integers $k\geq 1$. (See Theorem 4.2.1, of \cite{shara}). In the course of this paper we will assume all the potentials $q(t,x)$ are in $C^{\infty}([0,T]\times\mathcal{M})$, so the X-ray transform for time dependent potentials is well-defined. However, we only use the assumption $q(t,x)\in C^1([0,T]\times\mathcal{M})$ in the proofs (which could possibly be lowered to $C^1([0,T]; C^0(\mathcal{M}))$). If we assumed the potentials were time independent we could reduce the assumptions on the regularity of the potentials in view of \cite{shara}. 

 The goal of this paper is to recover the X-ray transform for a time dependent potential from the dynamical Dirichlet-to-Neumann map. If we allow a diffeomorphism $\Phi: \mathcal{M}\rightarrow \mathcal{M}$ such that $\Phi|_{\partial\mathcal{M}}=\mathcal{I}d$, then we have $\Lambda_{\Phi^*g,q}=\Lambda_{g,q}$ where $\Phi^*g$ is the pullback of the metric $g$ under $\Phi$. Naturally all of the results will be formulated modulo gauge invariance. We therefore consider the metric $g$ to be fixed for the rest of this paper. We let the length of the longest maximal geodesic in $\mathcal{M}$ be denoted as diam$_g(\mathcal{M})$. We prove a type of conditional H\"older stability estimate which is related to the one for simple manifolds and time independent potentials in \cite{M}. The first theorem is then:
\begin{thm}\label{main}
For any compact Riemannian manifold satisfying the admissibility criterion above, and any initial conditions $(x,\omega)\in \partial_+\mathcal{SM}$, there exists a finite $\epsilon_0>0$, such that the condition
\begin{align}
\norm{\Lambda_{g,q_1}-\Lambda_{g,q_2}}_{H_0^1\rightarrow L^2}<\epsilon_0
\end{align}
implies 
\begin{align}
 |I_{x,\omega}(q)| \leq C\norm{\Lambda_{g,q_1}-\Lambda_{g,q_2}}_{H_0^1\rightarrow L^2}^{\beta}
\end{align}
holds. As before, $I_{x,\omega}$ denotes the geodesic X-ray transform associated to ($\mathcal{M},g$) and $q_1(t,x),q_2(t,x)\in C^{\infty}([0,T]\times\mathcal{M})$. The constant $C$ depends only on the metric $g$, the $C^1([0,T]\times\mathcal{M})$ norm of the potential $q(t,x)=q_1(t,x)-q_2(t,x)$, $T$ and diam$_g(\mathcal{M})$. We also assume $T-\mathrm{diam}_g(\mathcal{M})\geq 4\epsilon_0^{\frac{1}{2\alpha n}}$ for some $\alpha\in \mathbb{R}, \alpha>1$ $n=\dim(\mathcal{M})$, so that $T-diam_g(\mathcal{M})$ is not too large. The number $\beta$ is in the interval $(0,1)$. 
\end{thm}
We notice here that Theorem \ref{main} does not require full boundary data. 

Theorem \ref{main} can be strengthened by applying the relatively new result by Uhlmann and Vasy \cite{vasy}. Let $\mathcal{M}_c$ be a strictly geodesicaly convex subset of the manifold $\mathcal{M}$. Let $\rho\in C^{\infty}(\mathcal{M})$ be a  global defining function of the boundary $\partial\mathcal{M}_c$, as considered as a function on $\mathcal{M}$. We consider the restriction of the X-ray transform to strictly geodesically convex subsets of $\mathcal{M}$. If we have $\mathcal{O}\subset  \mathcal{M}_c$ is an open set, we call the geodesic segments of the metric $g$ which are contained in $\mathcal{O}$ with endpoints on $\partial\mathcal{M}_c$, $\mathcal{O}$ local geodesics and we denote this collection as $\mathcal{M}_{\mathcal{O}}$ as in \cite{vasy}. The local geodesic transform of a function $f$ is defined on $\mathcal{M}_c$ as the collection of the integrals $f$ along the geodesics in $\mathcal{M}_{\mathcal{O}}$, that is the restriction of the X-ray transform to $\mathcal{M}_\mathcal{O}$.  

The main result of the paper \cite{vasy} can be formulated as below:
\begin{thm}\label{vasy}[Uhlmann and Vasy]
Let dim$\mathcal{M}=n\geq 3$. If we have that $\forall p\in \partial\mathcal{M}, \exists h(p)\in C^{\infty}(\mathcal{M})$ a function such that $h(p)=0$ and $dh=-d\rho$, and for $c$ sufficiently small with $\mathcal{O}_{p}=\{h(p)>-c\}\cap \overline{\mathcal{M}}_c$, the local geodesic transform is injective on $H^s(\mathcal{O}_p)$. Furthermore let $H^s(\mathcal{SM}|_{\mathcal{O}_p})$ be the restriction of elements in $H^s(\mathcal{SM})$ to $\mathcal{SM}|_{\mathcal{O}_p}$. For $F>0$, we define a weighted Sobolev class as follows:
\begin{align}
H^s_F(\mathcal{O}_p)=\exp(F/h+c)H^s=\{f\in H^s_{loc}(\mathcal{SM}|_{\mathcal{O}_p}): \exp(-F/(h+c)f\in H^s(\mathcal{O}_p)\}
\end{align}
then we have for any $s\geq 0$, there exists a constant $D$ such that $\forall f\in H^s_F(\mathcal{O}_p)$, 
\begin{align}
\norm{f}_{H^{s-1}_F(\mathcal{O}_p)}\leq D\norm{If|_{\mathcal{SM}|_{\mathcal{O}_p}}}_{H^s(\mathcal{O}_p)}.
\end{align}
\end{thm}
The authors \cite{vasy} consider domains $\mathcal{M}_c$ with equipped with a function $\rho: \overline{\mathcal{M}_c}\rightarrow[0,\infty)$ whose level sets $\Sigma_s=\rho^{-1}(s), s<S$ are strictly convex. 
Their theorem has the following global injectivity result as a corollary
\begin{cor}\label{ginj}[Uhlmann and Vasy]
For $\mathcal{M}_c$ and $\rho$ as above if the complement of $\bigcup\limits_{s\in[0,S)}\Sigma_s$  has empty interior  the global geodesic X-ray transform is injective on $H^s(\mathcal{M}), \forall s>n/2$. 
\end{cor}

As a consequence of combining the main theorems and the theorems in \cite{vasy} using compactness and Sobolev embedding, we have the following corollaries

\begin{cor}[Consequence of Theorem \ref{main}]
Let $E$ be subset of $\mathcal{SM}|_{\mathcal{O}_p}$, and $E_x$ be the associated set of $x$ such that $E_x\subset \mathcal{O}_p$. Let $I^E_{x,\omega}$ denote the geodesic ray transform restricted to $E$. In other words, the  geodesics which are being integrated over start and end on $E$. Let $q(x)=q_1(x)-q_2(x)\in H^{s-1}_F(\mathcal{O}_p)$ then we have for any $s\geq 0$, there exists a constant $D$  depending only on the metric $g$, the $C^1(\mathcal{M})$ norm of the potential $q(x)$, $T$, and diam$_g(E_x)$ a number $\beta\in(0,1)$ such that 
\begin{align}
\norm{q}_{H^{s-1}_F(\mathcal{O}_p)}\leq D\norm{\Lambda_{g,q_1}-\Lambda_{g,q_2}}_{H^1_0([0,T]\times E_x)}^{\beta}.
\end{align}
\end{cor}
\begin{proof}
If we consider $E_x$ to be a compact embedded submanifold of $\mathcal{M}$ then the result follows immediately from Theorem \ref{main} and Uhlmann and Vasy's result (Theorem \ref{vasy}). 
\end{proof}

\begin{cor}[Consequence of Theorem \ref{main}]
Assuming the potentials are time independent, if 
\begin{align}
\Lambda_{g,q_1}=\Lambda_{g,q_2}
\end{align}
then $q_1(x)-q_2(x)=0$ for all manifolds $\mathcal{M}=\mathcal{M}_c$ satisfying the assumptions in Corollary \ref{ginj} with the potentials must lie in the appropriate $C^s(\mathcal{M})$ as dictated by the dimension of the manifold.  
\end{cor}
\textbf{Remark}: This corollary was also observed in \cite{msd}. 

\section{Gaussian Beams}
The work here differs from the previous studies in \cite{sjo}, \cite{DDSF}, \cite{stefanovstable}, because we are not using the geometric optics Ansatz. Instead we will take our cue from the article by Ralston \cite{Ralston} and the book \cite{LKK}, and use a complex phase Ansatz called the Gaussian beam Ansatz. 
We start with the construction of the Gaussian beams, and show that we can solve the initial boundary value problem to a high degree of accuracy on the manifold $\mathcal{M}$, in a neighborhood of the curve $x(t)\in \mathcal{M}$, depending continuously on time, $t$. This curve $x(t)$ is actually a geodesic $\gamma(t)$ (see the Appendix for a proof in Fermi coordinates). We have the following theorem which begins to make the construction more precise. For the rest of this paper, we take $\lambda$ to be a scalar.

\begin{thm}\label{gaussians}
Let $d_g(\cdot,\cdot)$ denote the distance function associated to the Riemannian metric, $g$. Let $\lambda$ be a scalar which is our asymptotic parameter. Whenever $q(t,x)\in C^N([0,T]\times \mathcal{M})$ then for any finite $N$, we can construct nonzero functions $a_j(t,x)\in H^1([0,T],L^2(\mathcal{M}))$, $j=0,. . . ,N$ and $\psi(t,x)\in C^N([0,T]\times \mathcal{M})$ independent of $\lambda$, such that if we let 
\begin{align}\label{gb}
U^N_{\lambda}(t,x)=\sparen{\frac{\lambda}{\pi}}^{\frac{n}{4}}\exp(i\lambda\psi(t,x))\sum\limits_{j=0}^N\sparen{\frac{i}{\lambda}}^{j}a_j(t,x)
\end{align}
then 
\begin{align}
\sup\limits_{x\in\mathcal{M}, t\in [0,T]}|(\Box_g+q(t,x))U_{\lambda}^N(t,x)|\leq C\lambda^{-N+\frac{n}{4}}.
\end{align}
The coefficients, $a_j(t,x)$, are the amplitudes and $\psi(t,x)$ is a complex-valued phase function. We refer to $U_{\lambda}^N(t,x)$ as our formal gaussian beam of order $N$. We will see that it is essential in our construction $\psi(t,x)$ has positive definite imaginary part, by which we mean for any compact subset of $[0,T]\times \mathcal{M}$:
\begin{align}\label{posdef}
&\Im\psi(t,x(t))=0 \\&
\Im\psi(t,x)\geq C(t)d^2_g(x,x(t)) \nonumber
\end{align}
where again, $x(t)$ is a curve in $\mathcal{M}$ which depends continuously on $t$, and $C(t)$ is a continuous positive function $\forall t\in [0,T]$. The constants are uniform for any compact subset of $[0,T]\times\mathcal{M}$.  
\end{thm}
\begin{proof}
This construction is done in numerous places, for example see \cite{LKK} and \cite{Ralston}. We sketch it quickly because we will need the form of the phase for later computations. In order to find the phase and amplitudes, we substitute the asymptotic expansion of the Gaussian beam (\ref{gb}) into the wave equation to obtain
\begin{align}\label{sub}
(\Box_g+q(t,x))U_{\lambda}^N(t,x)=\sparen{\frac{\lambda}{\pi}}^{\frac{n}{4}}\exp\sparen{i\lambda\psi(x,t)}\sum\limits_{j=0}^{N+2}\sparen{\frac{i}{\lambda}}^{j-2}c_j(t,x).
\end{align}
Examining highest order terms in $\lambda$ first, we see that the phase function $\psi(t,x)$ must satisfy the eikonal equation
\begin{align}\label{eikonal}
(\psi_t)^2-g^{kl}(x)\psi_{x_k}\psi_{x_l}=0.
\end{align}
For the rest of this paper, we set 
\begin{align*}
h^2(x,\psi_x)=g^{kl}(x)\psi_{x_k}\psi_{x_l}.
\end{align*}
The solutions of the equation 
\begin{align*}
\psi_t\pm h(x,\psi_x)=0
\end{align*}
to high order along a single curve, $(t,x(t))$ in space time are central to the construction of the Gaussian beam. We want to find a phase which satisfies (\ref{eikonal}) to high order. Since the two cases are essentially the same, we consider trying to solve

\begin{align}\label{hflow}
\psi_t = h(x,\psi_x).
\end{align}
Let $h(x,p)$ be defined as 
\begin{align}
h(x,p)=\sqrt{g^{kl}(x)p_kp_l}
\end{align}
We claim $(x(t),\omega(t))$ with $(x(0),\omega(0))\in \partial\mathcal{SM}^+$
\begin{align}\label{systemomega}
\frac{dx(t)}{dt}=-h_{p}(x(t),\omega(t)) \quad \frac{d\omega(t)}{dt}=h_x(x(t),\omega(t))
\end{align}
is a curve which will allow such a construction. By the fundamental theorem of ordinary differential equations, this system is well posed for $t\in [0,T]$ with $T<\infty$. We chose the phase function to be real valued along the curve $x=x(t)$, and we fix the initial value of the phase function as
\begin{align}\label{icephase}
\psi(0,x)=i|x-x_0|^2/2+(x-x_0)\cdot \omega_0.
\end{align}
We will show we can write the phase function in a Taylor series with the first two terms given by:
\begin{align*}
\psi^1(t,x)=(x-x(t))\cdot\omega(t) \qquad \psi^2(t,x)=M_{lj}(t)(x-x(t))^{l}(x-x(t))^{j} 
\end{align*}
where $M(t)$ is a matrix such that $\Im M(t)$ is positive definite. 

Working backwards, if we differentiate the equation (\ref{hflow}) we obtain the following relations
\begin{align}\label{systempsi}
&\psi_{tx_j}-h_{p_i}(x,\psi_x)\psi_{x_ix_j}=h_{x_j}(x,\psi_x)\\ \nonumber
&\psi_{tt}-h_{p_i}(x,\psi_x)\psi_{x_it}=0\\ \nonumber
&\psi_{tx_jx_k}-h_{p_i}(x,\psi_x)\psi_{x_ix_jx_k}=\\ \nonumber
&h_{x_jx_k}(x,\psi_x)+h_{x_jp_i}(x,\psi_x)\psi_{x_ix_k}+h_{x_lp_i}(x,\psi_x)\psi_{x_ix_j}+h_{p_ip_m}(x,\psi_x)\psi_{x_ix_j}\psi_{x_mx_k}.
\end{align}
To simplify this set of relations we consider the matrices $A$, $B$ and $C$ which are defined with entries as follows:
\begin{align}\label{matrices}
& A_{j}^i=\{h_{x_ix_j}(x(t),\omega(t))\} \\ \nonumber
& B_j^i=\{h_{x_ip_j}(x(t),\omega(t))\}  \\ \nonumber
& C_j^i=\{h_{p_ip_j}(x(t),\omega(t))\}  \nonumber
\end{align}
If we set $\nabla_x\psi(t,x(t))=\omega(t)$ then the equations we know that the phase must satisfy (\ref{systempsi}) along the path $\{(t,x(t)): 0\leq t\leq T\}$ become
\begin{align}\label{bich}
&\frac{dx(t)}{dt}=-h_{p}(x(t),\omega(t)) \qquad \frac{d\omega(t)}{dt}=h_x(x(t),\omega(t)) \qquad \frac{d\psi}{dt}(t,x(t))=0 
\\ \label{ricatti}
& \frac{dM}{dt}=A+BM+MB^t+MCM 
\end{align}
The last equation (\ref{ricatti}) is a matrix Riccati equation associated to (\ref{bich}). It is a non-linear equation which is not always well-posed. From the equation $\psi_t(t,x(t))=0$ in (\ref{bich}), and the initial condition, this implies $\psi(t,x(t))=x(t)\cdot \omega(t)$ and $\psi_1(t)=\omega(t)$ as claimed. The crucial choice is therefore the Hessian, $M(t)$ which is associated to the second order terms in $(x-x(t))$. We chose the initial condition $M(0)=iI$. We also associate the matrices $Y(t)$ and $N(t)$ to the Hessian $M(t)$. Now we let $Y(t)$ and $N(t)$ satisfy the following system: 
\begin{align}\label{systemY}
&\frac{dY}{dt}=-B^tY-CN \qquad \frac{dN}{dt}=AY+BN \\& 
(Y(0),N(0))=(I,iI) \nonumber
\end{align}
We claim whenever $(Y(t),N(t))$ is a solution to (\ref{systemY}), then $Y(t)$ is invertible, and the solution $M(t)=N(t)Y^{-1}(t)$ to (\ref{ricatti}) exists for all bounded time intervals, if and only if $M(t)$ is positive definite. With the given initial conditions (\ref{icephase}), this is equivalent to the claim we can find a phase satisfying the condition (\ref{posdef}).  

We proceed to prove the claim by assuming that $Y(t)v=0$. If this is true then, 
\begin{align*}
0=\langle(Y(t)v,N(t)v),(Y(t)v,N(t)v\rangle_{\mathbb{C}}=\langle(v,iv),(v,iv)\rangle_{\mathbb{C}}=-2i|v|^2
\end{align*}
which implies $v=0$. Therefore $M(t)=N(t)Y(t)^{-1}$ is well defined-for all $t$ in a bounded interval. If we let $v_0=Y(t)^{-1}v$ and $w_0=Y(t)^{-1}w$, we have 
\begin{align*}
&w\cdot M(t)v-v\cdot M(t)w=Y(t)w_0\cdot N(t)v_0-Y(t)v_0\cdot N(t)w_0=\\&\nonumber
=\langle(Y(t)w_0,N(t)w_0),(Y(t)v_0,N(t)v_0)\rangle_{\mathbb{C}}=\langle(w_0,iw_0),(v_0,iv_0)\rangle_{\mathbb{C}}=0
\end{align*}
from which it follows that $M(t)=M(t)^t$. Similarly, we can see  
\begin{align*}
&v\cdot\overline{M(t)v}-\overline{v}\cdot M(t)v=\langle(Y(t)v_0,N(t)v_0),(Y(t)v_0,N(t)v_0)\rangle_{\mathbb{C}}=\\& \nonumber
\langle(v_0,iv_0),(v_0,iv_0)\rangle_{\mathbb{C}}=-2i|v_0|^2
\end{align*}
which proves that $\Im{M}(t)$ is positive definite so our claim (\ref{posdef}) is proved. 

Now we proceed to find the coefficients $a_j(t,x)$ of the beam.  The substitution (\ref{sub}) gives $c_j(t,x)$ is of the form:
\begin{align}\label{recursion}
&c_j(t,x)=\\& \nonumber \sparen{(\psi_t(t,x))^2-g^{kl}(x)\psi_{x_k}(t,x)\psi_{x_l}(t,x)}a_j(t,x)-La_{j-1}(t,x)+\sparen{\Box_g+q(t,x)}a_{j-2}(t,x)
\end{align}
where $j=0, . . ,N+2$ and initially $a_{-1}\equiv a_{-2}\equiv 0$. The linear operator $L$ is the transport operator which acts on functions $a(t,x)$ in the following manner
\begin{align*}
La=2\psi_t a_t-2g^{kl}\psi_{x_k}a_{x_l}+(\Box_g\psi)a
\end{align*}
Now, we see that in order for the the Ansatz to satisfy the PDE to high order, each $c_j$ for $j=0, . . . ,N+1$ must vanish to order $2(N+2-j)$ along the nul-bicharacteristic curves, which correspond to $x=x(t)$. Therefore, it is natural to consider $a_j(t,x)$ as a sum of homogeneous polynomial with respect to $x-x(t)$ as well, so we Taylor expand
\begin{align}\label{taylor}
a_j(t,x)=\sum\limits_{l\geq 0}a_{j,l}(t)(x-x(t))^l.
\end{align} 
From this identity we can match up term in our Taylor series expansion. Combining (\ref{bich}) and (\ref{recursion}), we obtain a differential equation for $a_{j,l}(t)$;
\begin{align}\label{diffeq}
\frac{d}{dt}a_{j,l}(t)+r(t)a_{j,l}(t)=F_{j,l}(t) 
\end{align} 
The right hand side is a homogenous polynomial of order $l$ in $x-x(t)$ which depends on $a_{l,k}(t)$ and $\psi_{k}$ where $k\leq l+2, r<j$ \cite{LKK}. The factor $r(t)$ comes from computing $\Box_g\psi$ along the curves (\ref{bich}). One sees that
\begin{align*}
\Box_g\psi=(\psi_t)_t-(\psi_t h_{p_i})_{x_i}-hh_{p_i}\frac{g_{x_i}}{2g}=&\\ h\sparen{\frac{g_t}{2g}-tr(B+CM)}=h\sparen{\frac{g_t}{2g}}+\mathrm{tr}\sparen{\frac{dY}{dt}Y^{-1}}.
\end{align*}
Using the identity 
\begin{align*}
\frac{d}{dt}\ln(\det R(t))=tr\sparen{\frac{dR(t)}{dt}R(t)^{-1}}
\end{align*}
we  obtain ordinary differential equations defining $a_{j,l}(t)$ as follows
\begin{align*}
\frac{d}{dt}a_{j,l}(t)+\sparen{\frac{1}{4}\frac{d}{dt}\ln[(\det(Y(t))^2g(t)]}a_{j,l}(t)=F_{j,l}(t)
\end{align*}
Solutions to these equations are given by 
\begin{align}\label{coeffeq}
a_{j,l}(t)=\sigma(t)\sparen{a_{j,l}(0)+\int\limits_0^t\sigma^{-1}(s)F_{j,l}(s)\,ds}
\end{align}
where 
\begin{align*}
\sigma(t)=\sparen{\frac{\det Y(0)}{\det Y(t)}}^{\frac{1}{2}}\sparen{\frac{g(x(0))}{g(x(t))}}^{\frac{1}{4}}.
\end{align*}
The theorem has the following simple corollaries which we will use later
\begin{cor}\label{size}
In local coordinates, we can write
\begin{align*}
a_{0}(t,x)=\sparen{\frac{\det Y(0)}{\det Y(t)}}^{\frac{1}{2}}\sparen{\frac{g(0)}{g(x(t))}}^{\frac{1}{4}}a(0,x)+\mathcal{O}(|x-x(t)|)
\end{align*}
\end{cor}
We can easily compute the first few terms given by (\ref{coeffeq}). Since we know that $F_{0,0}=0$, we compute
\begin{align}\label{a0}
a_{0,0}(t)=a_{0,0}(0)\sigma(t).
\end{align}
\end{proof}
The second corollary is:
\begin{cor}\label{expsize}
Let $\psi(t,x)$ correspond to a zeroth order beam. Let $\sim$ denote the equivalence relation bounded above and below by, then we have
\begin{align*}
\exp(-2\lambda\Im\psi(t,x))\sim\exp(-\lambda Cd^2_g(x,x(t))).
\end{align*}
where $C$ is independent of $\lambda$.  As a consequence if we let $A$ denote the set 
\begin{align*}
A=\{x: d_g(x,x(t))>\lambda^{-(\frac{1}{2}-\sigma)},\,\,\, 0\leq t\leq T\}  \qquad \sigma>0, \sigma \in\mathbb{R}
\end{align*}
then since $2\Im\psi(t,x)\sim d^2_g(x,x(t))$, $\exp(-2\lambda\Im\psi(t,x))$ is exponentially decreasing in $\lambda$ for all $x\in A$. 
\end{cor}
\begin{proof}
We need only observe that $M(t)$ is a bounded and positive definite matrix. From the form of the phase functions in Theorem \ref{gaussians}, the desired result follows. 
\end{proof}

We need another Lemma which comes from \cite{steinreal}, cf. Lemma 4.4. in Chapter 2
\begin{lem}\label{goodkernel}
Let $K_{\delta}(x)$ be a postive function depending on $\delta$ a small parameter and $x$, satisfy the following conditions
\begin{enumerate}
\item $\int\limits_{\mathbb{R}^n}K_{\delta}(x)\,dx=1$
\item $\forall \eta>0$ we have that 
\begin{align*}
\int\limits_{|x|>\eta}|K_{\delta}(x)|\,dx\rightarrow 0
\end{align*}
as $\delta\rightarrow 0$. 
\end{enumerate}
then we have $K_{\delta}\ast f\rightarrow f$ as $\delta\rightarrow 0$ if $f(x)\in C_c(\mathbb{R}^n)$.  
\end{lem} 
\begin{proof}
By continuity, $\forall \epsilon>0$, $\exists \eta>0$ such that $|y|<\eta$ implies $|f(x-y)-f(x)|<\epsilon$. We can then estimate
\begin{align}\label{estimate}
&|K_{\delta}\ast f-f(x)|\leq 
|f(x-y)-f(x)|\int\limits_{|y|\leq \eta}|K_{\delta}(y)|\,dy+\norm{2f(x)}_{C^0(\mathbb{R}^n)}\int\limits_{|y|> \eta}|K_{\delta}(y)|\,dy \\& \nonumber<\epsilon\sparen{\norm{2f(x)}_{C^0(\mathbb{R}^n)}+1}
\end{align}
whenever $\delta$ is sufficiently small. We could expand the proof to include $\forall f(x)\in L^1_{loc}(\mathbb{R}^n)$ but then the result is almost everywhere convergence as in \cite{steinreal}
\end{proof}

A corollary to the Lemma is:
\begin{cor}\label{cutoff}
There exists a cutoff function $\chi_{\epsilon_1}(x)\in C^{\infty}(\mathbb{R}^n)$ such that for $\epsilon_1>0$
\begin{align}
&\chi(x)=0 \qquad \mathrm{if} \qquad x\in \{x: d_g(x,x(t))>2^{\frac{1}{n}}\epsilon_1^{\frac{1}{2n\alpha}}\quad 0\leq t\leq T\} \\& \nonumber
\chi(x)=1  \qquad \mathrm{if} \qquad x\in \{x: d_g(x,x(t))<\epsilon_1^{\frac{1}{2n\alpha}}\quad 0\leq t\leq T\} \nonumber
\end{align}
and also for $m\in \mathbb{N}$. 
\begin{align}
\sup\limits_{x\in \mathbb{R}^n}|\nabla_g^{m}(\chi(x))|<2\epsilon_1^{-\frac{m}{2\alpha}}
\end{align}
for some constant $C$ which is independent of $\epsilon_1$, $n=\mathrm{dim}(\mathcal{M})$ and $\alpha\in\mathbb{R}$ is such that $\alpha>1$
\end{cor}
\begin{proof}
This is a consequence of Lemma \ref{goodkernel}. We take a step function which is $1$ on $\{x: d_g(x,x(t))<\epsilon_1^{\frac{1}{2n\alpha}}, 0\leq t\leq T\}$, and convolve it with a compactly supported good kernel while taking $\delta\rightarrow 0$. We consider $\mathcal{M}$ to be an embedded submanifold of $\mathbb{R}^n$, and points, $x$ not on $\mathcal{M}$ inherit the topology of the manifold if they are close by, so the gradient estimate makes sense. This possible because the manifold is smooth, so a Taylor series expansion of the metric exists in a small neighborhood.    
\end{proof}

\section{Gaussian Beams from the Boundary}
We would like to construct Gaussian beam solutions from initial data on the boundary of the manifold which pass through the interior of the manifold and are concentrated along geodesic curves in space-time. By knowing the collection of Dirichlet-to-Neumann maps, we can select any function $f_{\lambda}(t,x)$ as data for our initial boundary value problem:
\begin{align}\label{ivpf}
&(\Box_g+q(t,x))u(t,x)=0 \qquad\qquad \mathrm{on} \qquad (0,T)\times \mathcal{M}
\\&u(t,x)|_{t=0}=\partial_t u(t,x)|_{t=0}=0 \qquad \mathrm{in} \qquad \mathcal{M}, \nonumber
\\&u(t,x)=f_{\lambda}(t,x) \qquad \qquad \quad \mathrm{on} \qquad (0,T)\times \partial\mathcal{M} \nonumber
\end{align} 
As a direct result of Theorem \ref{gaussians}, we have the following corollary:
\begin{cor} \label{boundarybeams}
Given $(x_0,\omega_0)\in  \partial\mathcal{SM}^+$ and $t_0>0$, we can build build any $0^{th}$ order Gaussian beam, which we denote as $U_{\lambda}(t,x)$ satisfying the following requirements:
 \begin{enumerate}
 \item $a_0(t_0,x_0)=1$
 \item $\nabla_x\psi(t_0,x_0)=-\omega_0$ \qquad 
 \item $\psi_t(t_0,x_0)=1$.
 \item $\mathrm{supp}(U_{\lambda}(t_0,x))\subset B_{\epsilon_1^{\frac{1}{2n\alpha}}}(t_0,x_0)$
 \end{enumerate}
\end{cor} 
where $0<\epsilon_1<1$ is a small positive number. 
We further claim   
\begin{lem}\label{gcutoff}
Let $t_0=2\epsilon_0^ {\frac{1}{2n\alpha}}>0$ and $T>\mathrm{diam}_g(M)+4\epsilon_0^ {\frac{1}{2n\alpha}} $, and set $\epsilon_1=\norm{\Lambda_{g,q_1}-\Lambda_{g,q_2}}_{H_0^1\rightarrow L^2}$, then there is a Gaussian beam approximation to the solution $u$ which takes the form
\begin{align*}
U_{\lambda}(t,x)\chi_{\epsilon_1}(t,x)
\end{align*}
where $U_{\lambda}(t,x)\chi_{\epsilon_1}(t,x)=0$ if $t\geq T$ and $t\leq 0$, $\forall x\in\mathcal{M}$
\end{lem}
\begin{proof}
From Corollary \ref{boundarybeams}, there exists a zeroth order Gaussian beam approximation to $(\Box_g+q(t,x))u(t,x)=0$ From a modification of Corollary \ref{cutoff} we know there exists a smooth function $\chi_{\epsilon_1}(t,x)$ such that $\chi_{\epsilon_1}(t,x)\in C^{\infty}([0,T]\times\mathbb{R}^n)$ and
\begin{align*}
& \chi_{\epsilon_1}(t,x)=1 \quad \mathrm{if} \quad  (t,x)\in \{(t,x):|t-r-t_0|+d_g(x,x(r))<\epsilon_1^{\frac{1}{2n\alpha}}, 0\leq r\leq T \}\\
& \chi_{\epsilon_1}(t,x)=0 \quad \mathrm{if} \quad  (t,x)\in \{(t,x): |t-r-t_0|+ d_g(x,x(r))>2\epsilon_1^{\frac{1}{2n\alpha}}, 0\leq r \leq T \}
\end{align*}
We claim that our desired Gaussian beam is of the form 
\begin{align*}
U_{\lambda}(t,x)\chi_{\epsilon_1}(t,x)
\end{align*}
The cutoff then forces $U_{\lambda}(t,x)\chi_{\epsilon_1}(t,x)=0$ when $t\leq 0$ and $t\geq T, \forall x\in \mathcal{M}$, if we select $U_{\lambda}(t,x)$ satisfying the conditions of Corollary \ref{boundarybeams}. The choice of cutoff is motivated by the one in \cite{M}, however the way we are building the solution is different. 
\end{proof}

We can now construct a true solution to the wave equation which is localized along the ray path $\{(t,x(t)): 0\leq t\leq T\}$.  The treatment of the error estimates is similar to a combination of \cite{DDSF} and Liu and Ralston \cite{liu}.
\begin{lem}\label{error}
There is a solution to (\ref{ivpf}) of the form  
\begin{align*}
u(t,x)=U_{\lambda}(t,x)\chi_{\epsilon_1}(t,x)+R_{\lambda}(t,x)
\end{align*}
where $R_{\lambda}(t,x)$ satisfies
\begin{align}\label{okay}
&(\Box_g+q(t,x))R(t,x)=\\& \nonumber (\Box_g+q(t,x))(u(t,x)-U_{\lambda}(t,x)\chi_{\epsilon_1}(t,x))\qquad \mathrm{on} \qquad (0,T)\times \mathcal{M}
\\&R(t,x)|_{t=0}=\partial_t R(t,x)|_{t=0}=0 \qquad \mathrm{in} \qquad \mathcal{M}, \nonumber
\\&R(t,x)=0\qquad \qquad \quad \mathrm{on} \qquad (0,T)\times \partial\mathcal{M} \nonumber
\end{align}
and also 
\begin{align}\label{errorestimates}
&\sup_{t\in [0,T]}\sparen{\lambda\norm{R_{\lambda}(t,x)}_{L^2(\mathcal{M})}+\norm{R_{\lambda}(t,x)}_{\dot{H}^1(\mathcal{M})}+\norm{\partial_tR(t,x)}_{L^2({\mathcal{M})}}}\leq C\epsilon_1^{-\frac{1}{\alpha}}
\end{align}
where $C$ is independent of $\lambda$ and $\epsilon_1$, but depends on the $C^1([0,T]; C^0(\mathcal{M}))$ norm of $q(t,x)$.
\end{lem}

\begin{proof}
We make the definition
\begin{align*}
f_{\lambda}(t,x)=U_{\lambda}(t,x)\chi_{\epsilon_1}(t,x)|_{(0,T)\times\partial\mathcal{M}}
\end{align*}
so that $R(t,x)$ satisfies the equation set (\ref{okay}). 

\textbf{Remark:} The set
\begin{align}
\{(t,x)\in \mathbb{R}^+_t\times\mathcal{M} :|t-r-t_0|+d_g(x,x(r))<\epsilon_1^{\frac{1}{2n\alpha}}, 0\leq r\leq T \}
\end{align}
is a tube in space-time which encases the ray path $\{(r+t_0,x(r)): 0\leq r \leq T\}$ . The construction is predicated on the idea that the ray path in space-time has no self-intersections. This is not true if the metric depends on time. Defining the cutoff on the above set ensures 
\begin{align*}
& U_{\lambda}(t,x)\chi_{\epsilon_1}(t,x): \mathbb{R}^+_t\times\mathcal{M}\rightarrow \mathbb{R} \\& \nonumber
 U_{\lambda}(t,x)\chi_{\epsilon_1}(t,x)|_{(t,x)\in (0,T)\times\partial\mathcal{M}}: \mathbb{R}^+_t\times\partial\mathcal{M}\rightarrow \mathbb{R} \\& 
 \end{align*}
so in fact $f_{\lambda}(t,x)\in C_c^{\infty}(\mathbb{R}_t^+\times\partial \mathcal{M})$, which is necessary for the domain of the operator $\Lambda_{g,q}$ to be well-defined. It also ensures $R(t,x)=0$ on $(0,T)\times \partial\mathcal{M}$. This is most easily seen in Fermi coordinates as in the the Appendix, where the first coordinate variable, $x_1$ is identified with arclength, which one can call $r$. 

Let us set 
\begin{align*}
k(t,x)=(\Box_g+q(t,x))(U_{\lambda}(t,x)\chi_{\epsilon_1}(t,x)) \qquad \mathrm{and} \qquad k_1(t,x)=\int\limits_0^t k(s,x)\,ds,
\end{align*}
computing the integrand we see 
\begin{align}\label{integrandk1}
&k_1(t,x)=\int\limits_0^t k(s,x)\,ds=\int\limits_0^t\sparen{\frac{\lambda}{\pi}}^{\frac{n}{4}}B(s,x)\exp(i\lambda\psi(s,x)) \,ds
\end{align} 
where we have 
\begin{align*}
B(s,x)=-\sparen{\lambda^2c_{0}(s,x)-i\lambda c_{1}(s,x)+c_{2}(s,x)}\chi_{\epsilon_1}(s,x)+(\Box_g+q(s,x))\chi_{\epsilon_1}(s,x)a(s,x)
\end{align*}
where for each $i$, the $c_{i}(s,x)$ are given by equation (\ref{recursion}). 
Now we use Theorem \ref{gaussians} to obtain 
\begin{align}\label{orderc}
&c_{0}(s,x)=\sparen{(\psi_s)^2-g^{kl}\psi_{x_k}\psi_{x_l}}a(s,x)=\mathcal{O}(d_{g}(x,x(s))^4) \\& \nonumber
c_{1}(s,x)=2\psi_sa_{s}(s,x)-2g^{ki}\psi_{x_k}a_{x_i}(s,x)+\Box_g\psi a(s,x)=\mathcal{O}(d_{g}(x,x(s))^2)  \\& \nonumber
c_{2}(s,x)=(\Box_g+q(s,x))a(s,x).
\end{align} 
We also see 
\begin{align*}
\partial^2_t\psi(t,x)=\mathcal{O}(1) \qquad \partial_t\psi(t,x)=\mathcal{O}(1)
\end{align*}
and 
\begin{align}\label{ibound}
\norm{B(t,x)}_{L^1([0,T];L^2(\mathcal{M}))}+\norm{\partial_tB(t,x)}_{L^1([0,T];L^2(\mathcal{M}))}\leq C\epsilon_1^{-\frac{1}{\alpha}}
\end{align}
where the constants depend on the metric $g$, diam$_g(\mathcal{M})$. (The size of $|\partial_t\chi(t,x)|$ is cancelled by size of the support since we are integrating in $L^2$). We let $C$ be a generic constant which depends only on the metric $g$, the $C^1([0,T]; C^0(\mathcal{M}))$ norm of the potential $q(t,x)$, and diam$_g(\mathcal{M})$. We integrate by parts once in $s$ using the bound \ref{ibound}
\begin{align*}
\int\limits_0^t\sparen{\frac{\lambda}{\pi}}^{\frac{n}{4}}\sparen{\frac{B(s,x)}{i\lambda\psi_s(s,x)}}\partial_s\exp(i\lambda\psi(s,x)) \,ds.
\end{align*}
 to obtain 
\begin{align}\label{c2}
\norm{k_1(t,x)}_{L^2((0,T)\times\mathcal{M})}\leq \frac{C\epsilon_1^{-\frac{1}{\alpha}}}{\lambda}
\end{align}
where we recall that $U_{\lambda}(t,x)\chi_{\epsilon_1}(t,x)=0$ $\forall x\in \mathcal{M}$ whenever $t\geq T$ or $t\leq 0$. We know that 
\begin{align*}
\int\limits_0^tR_{\lambda}(s,x)\,ds
\end{align*}
solves the hyperbolic equation \ref{okay} with inhomogeneous term 
\begin{align*}
F(t,x)=k_1(t,x)+\int\limits_0^t\sparen{q(t,x)-q(s,x)}R_{\lambda}(s,x)\,ds
\end{align*}
with $F(t,x)$ as in Lemma \ref{wellposed}. Now it follows from an application of Gronwall's inequality as in Lemma 4.1 in \cite{DDSF}, there are constants $C_1,C_2$ independent of $\lambda$ such that
\begin{align*}
\sup\limits_{t\in(0,T)}\norm{R_{\lambda}(t,x)}_{L^2(\mathcal{M})}\leq C_1\norm{k_1(t,x)}_{L^2((0,T)\times\mathcal{M})}\leq \frac{C_2\epsilon_1^{-\frac{1}{\alpha}}}{\lambda}
\end{align*}
Since we know also know from Lemma \ref{wellposed},
\begin{align*}
&\sup\limits_{t\in(0,T)}\norm{R_{\lambda}(t,x)}_{\dot{H}^1(\mathcal{M})}+\sup\limits_{t\in(0,T)}\norm{\partial_tR_{\lambda}(t,x)}_{L^2(\mathcal{M})}\leq \nonumber
\\& C\norm{(\Box_g+q(t,x))\sparen{U_{\lambda}(t,x)\chi_{\epsilon_1}(t,x)}}_{L^1([0,T];L^2(\mathcal{M}))}\leq C\epsilon_1^{-\frac{1}{\alpha}}
\end{align*}
by the estimates in Theorem \ref{gaussians} the result follows. In summary, the solution to $\partial_t^2u-\Delta_{g}u+qu=0$ in $\mathcal{M}\times (0,T)$ with $\partial_t u(t,x)|_{t=0}=u(t,x)|_{t=0}=0$ in $\mathcal{M}$ can be approximated by a Gaussian beam which equals $f_{\lambda}(t,x)$ on the boundary and vanishes for $t<<0$, and $t>> T$. 

\textbf{Remark}: Should we have chosen to build a higher order beam, $U_{\lambda}^N(t,x)$, we could have obtained an estimate on $R_{\lambda}(t,x)$ of the form
\begin{align}\label{errorhigher}
&\sup\limits_{t\in(0,T)}\norm{R_{\lambda}(t,x)}_{H^1(\mathcal{M})}+\sup\limits_{t\in(0,T)}\norm{\partial_tR_{\lambda}(t,x)}_{L^2(\mathcal{M})}\leq \nonumber
\\& C\norm{(\Box_g+q(t,x))\sparen{U^N_{\lambda}(t,x)\chi_{\epsilon_1}(t,x)}}_{L^1([0,T];L^2(\mathcal{M}))}\leq \frac{C\epsilon_1^{-\frac{1}{\alpha}}}{\lambda^{N}}
\end{align} 
but this would require a higher regularity assumption on the potential $q(t,x)$. 
\end{proof}
\section{Green's Theorem}
We let $\mathrm{div} X$ be the divergence of the vector field $X\in H^1(T\mathcal{M})$ on $\mathcal {M}$, so that in local coordinates, one may write
\begin{align*}
\mathrm{div} X=\frac{1}{\sqrt{\det g}}\partial_i(\sqrt{\det g}\alpha_i), \qquad X=\alpha_i\frac{\partial}{\partial x_i}.
\end{align*}
Whenever $X\in H^1(T\mathcal{M})$ we have the standard divergence formula
\begin{align*}
\int\limits_{\mathcal{M}}\mathrm{div} X \,d_gV =\int\limits_{\partial\mathcal{M}}\langle X, \nu \rangle \,d\sigma_g^{n-1}. 
\end{align*}
so that when $f\in H^1(\mathcal{M})$, Green's formula states 
\begin{align}\label{green}
\int\limits_{\mathcal{M}}\mathrm{div} X f \,d_gV=-\int\limits_{\mathcal{M}} \langle X,\nabla_g f \rangle_g \,d_gV+\int\limits_{\partial\mathcal{M}} \langle X,\nu \rangle f\,d\sigma_g^{n-1}.
\end{align}
If we let $f\in H^1(\mathcal{M})$ and $w\in H^2(\mathcal{M})$, then the following holds:
\begin{align*}
\int\limits_{\mathcal{M}}\Delta_g w f \,d_gV=-\int\limits_{\mathcal{M}}\langle \nabla_g w,\nabla_g f \rangle _g \,d_gV +\int\limits_{\partial\mathcal{M}}\partial_{\nu}wf\,d\sigma_g^{n-1}.
\end{align*}

\section{Stability Estimates: Green's Theorem}

The goal of this section is to prove Theorem \ref{main} as a sequence of Lemmas using Green's theorem (\ref{green}) and Gaussian beam solutions. From Theorem \ref{gaussians}, we know that there exist zeroth order Gaussian beam approximations $U_{\lambda}(t,x)$ and $W_{\lambda}(t,x)$ corresponding to the solution of the initial boundary problem with electric potential  $q_2(t,x)$:
\begin{align}\label{forward}
&(\Box_g+q_2(t,x))u_2(t,x)=0 \quad\,\,\, \mathrm{on} \qquad (0,T)\times \mathcal{M}
\\&u_2(t,x)|_{t=0}=\partial_t u_2(t,x)|_{t=0}=0 \qquad \mathrm{in} \qquad \mathcal{M} \nonumber
\\&u_2(t,x)=f_{\lambda}(t,x) \qquad \qquad \quad \mathrm{on} \qquad (0,T)\times \partial\mathcal{M} \nonumber
\end{align} 
and the initial boundary problem for the backward wave equation with electric potential $q_1(t,x)$
\begin{align}\label{backward}
&(\Box_g+q_1(t,x))u_1(t,x)=0 \quad\,\,\, \mathrm{on} \qquad (0,T)\times \mathcal{M}
\\&u_1(t,x)|_{t=T}=\partial_t u_1(t,x)|_{t=T}=0 \qquad \mathrm{in} \qquad \mathcal{M} \nonumber
\\&u_1(t,x)=f_{\lambda}(t,x) \qquad \qquad \quad \mathrm{on} \qquad (0,T)\times \partial\mathcal{M}. \nonumber
\end{align} 
We relabel so it is understood that $U_{\lambda}(t,x)$ and $W_{\lambda}(t,x)$ contain the necessary cutoffs. We are considering our boundary value data to have the same initial phase as constructed in Corollary \ref{boundarybeams}.  We know that
\begin{align*}
U_{\lambda}(t,x)-\sparen{\frac{\lambda}{\pi}}^{\frac{n}{4}}a_0(t,x)\exp(i\lambda\psi(t,x))=0
\end{align*}
Here the function $a_0(t,x) \in H^1(\mathbb{R},L^2(\mathcal{M}))$ satisfies the transport equations to leading order and the phase function $\psi(t,x)$ is the corresponding phase function satisfying the eikonal. It is an important point to notice that because of the form of the initial data 
\begin{align}\label{limiting}
U_{\lambda}(t,x)\overline{W_{\lambda}(t,x)}-\sparen{\frac{\lambda}{\pi}}^{\frac{n}{2}}|a_0(t,x)|^2\exp(-2\lambda\Im(\psi(t,x)))=0
\end{align} 

We start the proof of Theorem \ref{main} by relating our Gaussian beam solutions to the Dirichlet-to-Neumann maps of the potentials via Green's theorem (\ref{green}). We let $q_1(t,x)$ and $q_2(t,x)$ be real valued potentials. Recall we have set  
\begin{align*}
q_1(t,x)-q_2(t,x)=q(t,x).
\end{align*}
We will prove the following Lemma which is similar in spirt to Lemma 5.1 in \cite{DDSF}. 
\begin{lem}\label{step}
There exists constants $C_1$ and $C_2$, independent of $\lambda$, depending on the metric $g$, diam$_g\mathcal{M}$, and $\norm{q(t,x)}_{C^1([0,T];C(\mathcal{M}))}$, with $U_{\lambda}(t,x)$ and $W_{\lambda}(t,x)$ as defined above such that 
\begin{align}\label{stepone}
\sabs{\int\limits_0^T\int\limits_{\mathcal{M}}q(t,x)U_{\lambda}(t,x)\overline{W_{\lambda}(t,x)}\,dV_{g} dt}
\leq \frac{C_1\epsilon_1^{-\frac{1}{\alpha}}}{\lambda}
\end{align}
whenever $\lambda$ is sufficiently large. 
\end{lem}
In order to prove the estimate above we need to know more information about the size of the error terms $R_{\lambda}(t,x)$ and $R^W_{\lambda}(t,x)$ as defined by:
\begin{align*}
u_2(t,x)-U_{\lambda}(t,x)=R_{\lambda}(t,x)  \qquad \mathrm{and} \qquad u_1(t,x)-W_{\lambda}(t,x)={R}_{\lambda}^W(t,x). 
\end{align*}
Theorems \ref{gaussians} and Corollary \ref{boundarybeams} allow us to consider the size of the terms. With the error estimates given by Corollary \ref{boundarybeams}, we can complete the proof of Lemma \ref{step}.
\begin{proof}[Proof of Lemma \ref{step}]
We let $v$ be the solution to the initial boundary value problem
\begin{align*}
&(\Box_g+q_1(t,x))v(t,x)=0 \quad\,\,\, \mathrm{on} \qquad (0,T)\times \mathcal{M}
\\&v(t,x)_{t=0}=\partial_t v(t,x)|_{t=0}=0 \qquad \mathrm{in} \qquad \mathcal{M} \nonumber
\\&v(t,x)= f_{\lambda}(t,x) \qquad \qquad \mathrm{on} \qquad (0,T)\times \partial\mathcal{M} \nonumber
\end{align*} 
If we set $w(t,x)=v(t,x)-u_2(t,x)$, then we obtain 
\begin{align}\label{hyperbolic}
&(\Box_g+q_1(t,x))w(t,x)=q(t,x)u_2(t,x) \quad\,\,\, \mathrm{on} \qquad (0,T)\times \mathcal{M}
\\&w(t,x)|_{t=0}=\partial_t w(t,x)|_{t=0}=0 \qquad \mathrm{in} \qquad \mathcal{M} \nonumber
\\&w(t,x)=0 \qquad \qquad \quad \mathrm{on} \qquad (0,T)\times \partial\mathcal{M} \nonumber
\end{align} 
Because $q(t,x)u_2(t,x)\in L^1([0,T]; L^2(\mathcal{M}))$, by Lemma \ref{wellposed} we know
\begin{align*}
w(t,x)\in C([0,T];H_0^1(\mathcal{M}))\cap C^1([0,T]; L^2(\mathcal{M})).
\end{align*}
Using integration by parts and Green's theorem (\ref{green}), we obtain the integral identity: 
\begin{align}\label{greenexpansion}
&\int\limits_0^T\int\limits_{\mathcal{M}}(\Box_g+q_1(t,x))w(t,x)\overline{u_1(t,x)}\,d_g V \,dt =\\& \nonumber
\int\limits_0^T\int\limits_{\mathcal{M}}q(t,x)u_2(t,x)\overline{u_1(t,x)}\,d_gV\,dt=\int\limits_0^T\int\limits_{\partial\mathcal{M}}-\partial_{\nu}w(t,x)\overline{u_1(t,x)}\,d\sigma_g^{n-1}\,dt
\end{align}

We construct our formal Gaussian beam solutions, $U_{\lambda}$ and $W_{\lambda}$ as in Corollary \ref{boundarybeams}. This implies 
\begin{align*}
&\int\limits_0^T\int\limits_{\mathcal{M}}q(t,x)u_2(t,x)\overline{u_1(t,x)}\,d_{g}V\,dt=\int\limits_0^T\int\limits_{\mathcal{M}}q(t,x)U_{\lambda}(t,x)\overline{W_{\lambda}(t,x)}\,d_{g}V\,dt+\\& \nonumber
\int\limits_0^T\int\limits_{\mathcal{M}}q(t,x)U_{\lambda}(t,x)\overline{R^W_{\lambda}(t,x)}\,d_{g}V\,dt+\int\limits_0^T\int\limits_{\mathcal{M}}q(t,x)R_{\lambda}(t,x)\overline{W_{\lambda}(t,x)}\,d_{g}V\,dt+\\& \nonumber \int\limits_0^T\int\limits_{\mathcal{M}}q(t,x)R_{\lambda}(t,x)\overline{R^W_{\lambda}(t,x)}\,d_{g}V\,dt
\end{align*}
Each of the last three terms in the sum above is bounded by symmetry, since we have
\begin{align}\label{3}
\sabs{\int\limits_0^T\int\limits_{\mathcal{M}}q(t,x)U_{\lambda}(t,x)\overline{R_{\lambda}(t,x)}\,d_{g}V\,dt}\leq \frac{C\epsilon_1^{-\frac{1}{\alpha}}}{\lambda}
\end{align}
by (\ref{errorestimates}), and Corollary \ref{boundarybeams}. 
Examining the right hand side of (\ref{greenexpansion}) we see by the trace theorem and choice of initial data $f_{\lambda}(t,x)$ that 
\begin{align}\label{1}
\sabs{\int\limits_0^T\int\limits_{\partial\mathcal{M}}\partial_{\nu}w(t,x)\overline{u_1(t,x)}\,d\sigma_g^{n-1}\,dt} \leq &\\
\norm{f_{\lambda}(t,x)}_{H^1([0,T]\times \mathcal{M})}\norm{f_{\lambda}(t,x)}_{L^2([0,T]\times \mathcal{M})}\norm{\Lambda_{g,q_1}-\Lambda_{g,q_2}}_{H_0^1\rightarrow L^2}\leq &\\ \nonumber
C\lambda\norm{a_0(t,x)}^2_{H^1((0,T)\times\mathcal{M}))}\norm{\Lambda_{g,q_1}-\Lambda_{g,q_2}}_{H_0^1\rightarrow L^2}
\end{align}
Combining the norm estimates (\ref{3}), (\ref{1}), and substituting into (\ref{greenexpansion}), we obtain the desired result (\ref{replacement2}). 
\end{proof}
We can now make the step to replace the integral on the left hand side of inequality (\ref{stepone}). We make the definitions
\begin{align}
\int\limits_b^{\infty}\exp(-x^2)\,dx=\mathrm{erfc}(b)  \qquad \int\limits_0^{b}\exp(-x^2)\,dx=\mathrm{erf}(b)
\end{align}
We know that the exponential function admits the following asymptotics:
\begin{align}\label{reallylarge}
\mathrm{erfc}(b)= \frac{\exp(-b^2)}{2b} +\mathcal{O}\sparen{\frac{\exp(-b^2)}{b^3}}
\end{align}
from Example 4 on page 255 of \cite{bender}, whenever $b$ is sufficiently large.  
Using this definition, we need to show that the Gaussian beams act like good kernels. We claim: 
\begin{lem}\label{step21}
There exists constants $C_1,C_2>0$ independent of $\lambda$, depending on the metric $g$, $T$, diam$_g(\mathcal{M})$, and $\norm{q(t,x)}_{C^1((0,T)\times \mathcal{M})}$ such that
\begin{align}\label{replacement2}
\sabs{\int\limits_{0}^T\int\limits_{\mathcal{M}}q(t,x)\sparen{\frac{\lambda}{\pi}}^{\frac{n}{2}}\chi_{\epsilon_1}(t,x)|a_0(t,x)|^2\exp(-2\lambda \Im(\psi(t,x)))\,dV_{g}\,dt-\int\limits_{0}^{T}q(t,x(t))\,dt}\leq &\\ \nonumber \frac{C_1\lambda^{\sigma}\epsilon_1^{-\frac{1}{2\alpha}}}{\sqrt{\lambda}}+C_2\mathrm{erfc}(-\lambda^{2\sigma})
\end{align}
\end{lem}
The proof works regardless of the value of $\lambda$.
\begin{proof}[Proof of Lemma \ref{step21}]
We cite results from \cite{steinreal}, and Lassas et. al \cite{LKK} to obtain the desired theorem. From \cite{steinreal}, we see: 
\begin{lem}\label{H}
Let $h(t,x)\in C^1((0,T)\times O)$, where $O$ is an open subset of $\mathbb{R}^n$ and $B$ be a symmetric nonsingular matrix such that $\Re{B}\geq 0$, if $x(t)$ is a continuous curve defined in terms of $t$ in $O$, then we have the following uniform estimate
\begin{align}\label{rrep}
&\sabs{\sparen{\frac{\lambda}{\pi}}^{\frac{n}{2}}(\det B)^{\frac{1}{2}}\int\limits_{O}\exp\sparen{\langle-\lambda B(x-x(t)), (x-x(t)\rangle} h(t,x)\chi_{\epsilon_1}(t,x)\,dx-h(t,x(t))}< \\& \nonumber
\sparen{\frac{2\lambda^{\sigma}\epsilon_1^{-\frac{1}{2\alpha}}}{\sqrt{\lambda}}+4\mathrm{erfc}(-\lambda^{2\sigma})}\norm{h(t,x)}_{C^1((0,T)\times O)}
\end{align}
Here $\chi_{\epsilon_1}(t,x)$ has the same definition as in Corollary \ref{cutoff}, but with the Euclidean metric. 
\end{lem}
\begin{proof}
 The assumption that $h(t,x)$ is in $C^1([0,T]\times O)$ implies that $h(t,x)$ is locally uniformly Lipschitz continuous with Lipschitz constant $\norm{h(t,x)}_{C^1((0,T)\times O)}$. We set $\epsilon=\lambda^{\sigma-1/2}\norm{h(t,x)}_{C^1((0,T)\times O)}$ as in the proof of Lemma \ref{goodkernel}.  We know that for $\eta=\lambda^{\sigma-1/2}$, if $x$ is such that $|x-x(t)|<\eta$, by Corollary \ref{cutoff} (recall only differentiating in the transverse subset variables) this implies
\begin{align*}
\sabs{h(t,x)\chi_{\epsilon_1}(t,x)-h(t,x(t))\chi_{\epsilon_1}(t,x(t))}<2\frac{\lambda^{\sigma}\epsilon_1^{-\frac{1}{2\alpha}}}{\sqrt{\lambda}}\norm{h(t,x)}_{C^1((0,T)\times O)}
\end{align*}
Using Corollary \ref{expsize} and change of variables, we then obtain the bounds
\begin{align*}
&\sabs{\sparen{\frac{\lambda}{\pi}}^{\frac{n}{2}}(\det B)^{\frac{1}{2}}\int\limits_{O}\exp\sparen{\langle-\lambda B(x-x(t)), (x-x(t)\rangle} h(t,x)\chi_{\epsilon_1}(t,x)\,dx-h(t,x(t))}<\\& \nonumber \frac{2\lambda^{\sigma}\epsilon_1^{\frac{-1}{2\alpha}}}{\sqrt{\lambda}}\norm{h(t,x)}_{C^1((0,T)\times O)}\int\limits_{|y|\leq C\eta}\sparen{\frac{\lambda}{\pi}}^{\frac{n}{2}}\exp(-\lambda |y|^2)\,dy+\\& \nonumber
2\norm{h(t,x)}_{C^0((0,T)\times O)}\int\limits_{C\eta<|y|< \infty}\sparen{\frac{\lambda}{\pi}}^{\frac{n}{2}}\exp(-\lambda |y|^2)\,dy \leq \\& \nonumber
 \sparen{\frac{2\lambda^{\sigma}\epsilon_1^{\frac{-1}{2\alpha}}}{\sqrt{\lambda}}+ 4\mathrm{erfc}(-\lambda^{2\sigma})}\norm{h(t,x)}_{C^1((0,T)\times O)}
\end{align*}
Here we notice that normalization factor of $(\det B)^{1/2}$ makes the Gaussian kernel normalized to $1$ as in the proof of Lemma \ref{goodkernel}.
\end{proof}
 
If we consider $(0,T)\times\mathcal{M}$ as an embedded submanifold of $\mathbb{R}^{n+1}$, then we can accurately approximate the X-ray transform as 
\begin{align}\label{replacement}
&|\int\limits_{0}^{T}\int\limits_{\mathcal{M}}\sparen{\frac{\lambda}{\pi}}^{\frac{n}{2}}q(t,x)|a_0(t,x)|^2\chi_{\epsilon_1}(t,x)\exp(-2\lambda \Im(\psi(t,x)))\,dV_{g}\,dt-
\\& \int\limits_{0}^{T}\sparen{\det \Im M(t)}^{-\frac{1}{2}}|a_0(x(t))|^2q(t,x(t))\,dt|\leq \nonumber \\& \sparen{\frac{2T\lambda^{\sigma}\epsilon_1^{-\frac{1}{2\alpha}}}{\sqrt{\lambda}}+4T\mathrm{erfc}(-\lambda^{2\sigma})}\norm{q(t,x)a_0(t,x)}_{C^1((0,T)\times \mathcal{M})}
\end{align} 
A proof using local coordinates is also done in the Appendix.

\textbf{Remark:} We could lower the regularity assumption on the potential to $q(t,x)\in C^0([0,T]\times\mathcal{M})$ by using a modification of Lemma \ref{goodkernel} but the analysis is more difficult when the time interval is small since the kernel depends on $\lambda$, and we are minimizing $\lambda$ with respect to $\epsilon_1$ in the final step.

From Corollary \ref{size}, the size of $|a_0(x(t))|$ is given by
\begin{align*}
|a_0(x(t))|^2=\sparen{\frac{|\det Y(0)|}{|\det Y(t)|}}\sparen{\frac{|g(0)|}{|g(x(t))|}}^{\frac{1}{2}}|a_0(t_0,x_0)|^2
\end{align*}
In \cite{LKK}, Lemma 2.58, they derive the following identity
\begin{lem}\label{LKK}
The identity holds 
\begin{align*}
\sparen{\det \Im M(t)}|\det Y(t)|^2=C
\end{align*}
where the constant $C$ depends time $T$. 
\end{lem}
By our choice of initial data, we also have $|a(t_0,x_0)|=1$. We notice that is is okay to ignore the cutoff functions since they are equal $1$ on the curve $x(t)$.  Combining the two lemmas,  we obtain the desired result. This idea is similar to Section 7 of \cite{ks}.
\end{proof}

\begin{proof}[Proof of Theorem \ref{main}]
We  recall properties of the phase functions and \ref{limiting} which imply 
\begin{align*}
U_{\lambda}(t,x)\overline{W_{\lambda}(t,x)}=\sparen{\frac{\lambda}{\pi}}^{\frac{n}{2}}|a_0(t,x)|^2\exp(-2\lambda\Im(\psi(t,x))
\end{align*}
From the triangle inequality and Lemma \ref{step21}, we obtain
\begin{align*}
&\sabs{\int\limits_0^T q(t,x(t))\,dt} \leq \frac{C_1\epsilon_1^{-\frac{1}{\alpha}}}{\lambda}+2\lambda\epsilon_1\norm{a_0(t,x)}^2_{H^1((0,T)\times\mathcal{M}))}+ \\&C\sparen{\frac{2T\lambda^{\sigma}\epsilon_1^{-\frac{1}{2\alpha}}}{\sqrt{\lambda}}+4T\mathrm{erfc}(-\lambda^{2\sigma})}\norm{q(t,x)a_0(t,x)}_{C^1((0,T)\times \mathcal{M})}\nonumber
\end{align*} 
Because $\epsilon_1=\norm{\Lambda_{g,q_1}-\Lambda_{g,q_2}}_{H_0^1\rightarrow L^2}<\epsilon_0$,  the main result now follows from minimization in $\lambda$. The constant $C_1\leq C(T)\sparen{\norm{q(t,x)}_{C^1((0,T)\times \mathcal{M})}+\norm{a_0(t,x)}_{C^3((0,T)\times \mathcal{M})}}$ (which could be made more explicit) and the size of the other constants ensure that constant in the final estimate will not be too large, with some normalization. In order to see this, let $h(\lambda)$ be defined as follows
\begin{align}
h(\lambda)=\frac{C_1\epsilon_1^{-\frac{1}{\alpha}}}{\lambda}+\frac{C_3\lambda^{\sigma}\epsilon_1^{-\frac{1}{2\alpha}}}{\sqrt{\lambda}}+C_4\mathrm{erfc}(-\lambda^{2\sigma})+C_2\lambda\epsilon_1
\end{align}
so that $h(\lambda)$ is a positive function since $\lambda$ and $\epsilon_1$ are positive functions, and we assume the constants are positive. We notice that since $\epsilon_1<1$ that $h'(\lambda)=0$, $h''(\lambda)>0$ when for appropriate $C_1',C_2',$ and $C_4'$ all greater than $0$
\begin{align}\label{desired}
\sparen{\frac{C_1'\epsilon_1^{-\frac{1}{\alpha}}}{\lambda^2} +\frac{C'_3\lambda^{\sigma-1}\epsilon_1^{-\frac{1}{2\alpha}}}{\sqrt{\lambda}}+C_4'\lambda^{\sigma-1}\exp(-\lambda^{2\sigma})}=C_2\epsilon_1
\end{align}
We can conclude the result if $\lambda$ can be made to be of the form $\epsilon_1^{-1+l}$, where $l\in (0,1)$. H\"older stability happens when the minimum in $\lambda\sim\epsilon_1^{-1+l}$ where $l\in (0,1)$, and $\sigma\in (0,1/2)$ is fixed. Otherwise the term $\epsilon_1\lambda$ in $h(\lambda)$ cannot be bounded by some $\epsilon_1^{\beta}, \beta\in(0,1)$, and the same with $\epsilon_1^{-\frac{1}{\alpha}}\lambda^{-1}$, etc. (The result is predicated on the idea $x^{\beta_1}<x^{\beta_2}$, if $\beta_1, \beta_2\in (0,1)$, $\beta_2<\beta_1$, and $x\in (0,1)$) The upper bound and the lower bound are important due to the presence of both positive and negative powers of $\lambda$. This also forces $\alpha>1, 1-2l>1/\alpha$. We sketch why such a minimum is possible. We notice that the solution $\lambda_1$ to 
\begin{align}\label{over}
\frac{C'_3\lambda_1^{\sigma-1}\epsilon_1^{-\frac{1}{2\alpha}}}{\sqrt{\lambda_1}}=C_2\epsilon_1
\end{align}
undershoots the solution to (\ref{desired}), since the constants are all positive, in other words $h'(\lambda_1)<0$. Solving (\ref{over}), we claim that $\lambda_1=C\epsilon_1^{-1+l}$ where $C$ is independent of $\epsilon_1,\lambda$ and $\alpha>1$ can be found in terms of $l\in (0,1)$ and $\sigma$. Their relationship is given by
\begin{align}\label{cn}
\sparen{\frac{3}{2}-\sigma}(1-l)=1+\frac{1}{2\alpha}
\end{align}
hence the limiting behavior where $\sigma\rightarrow 1/2$ forces $\alpha\rightarrow \infty, l\rightarrow 0$. The value $\lambda_2=C\epsilon_1^{-1}$, overshoots the solution ($h'(\lambda_2)>0$) provided if we plugged in $\lambda_2$ to the left hand side of (\ref{over}) we had an inequality instead: 
\begin{align}
CC_1'\epsilon_1^{2-\frac{1}{\alpha}} +C^{-\sigma+1/2}C'_3\epsilon_1^{-\sigma+3/2-\frac{1}{2\alpha}}+C_4'C^{\sigma}\epsilon_1^{\sigma+1}<C_2\epsilon_1.
\end{align}
This inequality is satisfied if $\epsilon_1$ is small and $\alpha$ satisfies the relationship (\ref{cn}) for some $l\in (0,1)$. If $\lambda_3$ is such that $h'(\lambda_3)=0$, then by the intermediate value theorem $\lambda_3\in (\lambda_1,\lambda_2)=(C\epsilon_1^{-1+l},C\epsilon_1^{-1})$ has the desired form (again provided $\epsilon_0$ is small). Using the asymptotic behavior (\ref{reallylarge}) in \cite{bender}, we can obtain the result, that $h(\lambda)$ is always bounded by $C\epsilon_1^{\beta}$ provided $\epsilon_1$ is sufficiently small. 

\textbf{Remark:} We note that the choice of finite $\epsilon_0<<1$ in Theorem \ref{main}, ensures that $\lambda$ is large. This makes the replacement of the term $\norm{R_{\lambda}(t,x)}_{L^1([0,T]; L^2(\mathcal{M}))}$ by $\mathcal{O}(\lambda^{-1})$ feasible as in \cite{DDSF} and \cite{M}. The regime where $\epsilon_1<1$ is the only one that makes sense here. Further calculations could expand the range of feasible $\epsilon_0$, but $\epsilon_0<1$. Because we use $\epsilon_0$ sufficiently small, we could have instead used the estimate (\ref{reallylarge}) from the beginning. 
\end{proof}

\textbf{Remark:} If the time interval was such that $T-\mathrm{diam}_g(\mathcal{M})>>\epsilon_0$, we could have chosen a cutoff independent of $\epsilon_1$ so that the characteristic function does not have such a steep slope, which makes the analysis easier. 

\textbf{Remark:} We could have chosen to build higher order beams and use the good kernels Lemma \ref{goodkernel} from Stein. However, the limiting step to better stability estimates seems to be Lemma \ref{H}. We do not know how to make the error smaller than $\mathcal{O}(\lambda^{-1/2})$, without more assumptions on the form of the potential, such as making it lie in the space $C_0^2([0,T]\times\mathcal{M})$. The vanishing on the boundary condition would allow for an integration by parts argument similar to the one in Lemma 3.3.6 and Theorem 3.3.4 in H\"ormander \cite{hormander}. It also means that the error is exactly $\mathcal{O}(\lambda^{-1})$ without the exponential term, regardless of the value of $\lambda$ which makes the analysis easier. 

\section{Appendix: Fermi Coordinates, Proof of Lemma \ref{step21} in local coordinates}

Sometimes it is instructive to view calculations in local coordinates. This section follows the treatment on Fermi coordinates in \cite{ks} and also very closely \cite{mazzeo}, in an attempt to expose the difference to the elliptic cases examined in \cite{ks} and \cite{msd}. We would like to construct Gaussian beam solutions from initial data on the boundary of the manifold which are concentrated along geodesic curves in space time. We introduce Fermi coordinates in order to help with the construction.  Suppose $\Gamma$ is a geodesic in an $n$ dimensional Riemannian manifold, and we fix an arclength parametrization, $\gamma(r)$ of $\Gamma$. We consider geodesics, $\gamma(r)$, on $\mathcal{M}$ which start at $x\in \partial\mathcal{M}$ with initial velocity $\omega$ such that $(x,\omega)\in \mathcal{SM}^+$, so we may write $\gamma(r)=\gamma_{x,\omega}(r)$. We let the basis for the tangent space $T_{x}M$ be denoted by $\{\dot{\gamma}_{x,\omega}(0), v_2, . . ,v_n\}$. Following the book by Gray \cite{grey} and \cite{ks}, we fix a parallel orthogonal frame $E_2(r), . . ,E_{n}(r)$ for the normal bundle, $N\Gamma$ to $\Gamma$ in $\mathcal{M}$, which we translate along the geodesic curve. Let $x'=(x_2, . . ,x_{n})$, then this parallel transport process determines a system of coordinates, which are related to Fermi coordinates. We let $F$ be the map such that 
\begin{align*}
&F: (r,x')=(r,x_2,x_2, . . .,x_n) \mapsto \exp_{\gamma(r)}(x_2E_2+. . . +x_{n}E_{n}) \\ \nonumber
&F: \mathbb{R}^n\rightarrow \mathcal{M}
\end{align*}
where we have used the indices $j,k,l\in\{2, . . ,n\}$ and $\alpha,\beta,\delta\in \{1, . . ,n\}$ We also use $X_{\alpha}=F_{*}(\partial_{x_\alpha})$. We remark that $|x'|=\sqrt{x_2^2+. . . +x_n^2}$ is the geodesic distance from $x$ to $\Gamma$ and $\partial_{r}$ is the unit normal to the hypersurfaces $\{x: d(x,\Gamma)=C\}$ where $C$ is some constant. The set $\{x: d(x,\Gamma)=C\}$ we refer to as a geodesic tube. We will primarily be doing computations in a neighborhood of the tubes. The restriction to small geodesic tubes will aid in the computations done because of the form of the Riemannian metric in a neighborhood of the tubes. Indeed if we let  $p=F(r,0)$ and $q=F(r,x')$, and $|x'|^2=d(p,q)$ then we have, cf \cite{grey}, 
\begin{align*}
&g_{jk}(q)=\delta_{jk}+\frac{1}{3}g(R(X_s,X_j)X_l,X_k)x_sx_l+\mathcal{O}(|x'|^3)\\ 
&g_{1k}(q)=\mathcal{O}(|x'|^2)\\ \nonumber
&g_{11}(q)=1-g(R(X_k,X_0)X_l,X_l)_px_kx_l+\mathcal{O}(|x'|^3)\\ \nonumber
&\Gamma_{\alpha \beta}^{\delta}=\mathcal{O}(|x'|)\\ \nonumber
&\Gamma_{11}^k=-\sum\limits_{j=1}^n g(R(X_k,X_0)X_j,X_0)x_j+\mathcal{O}(|x'|^2) \nonumber
\end{align*}
where the Schwarz Christoffel symbols are given by 
\begin{align*}
\Gamma_{\alpha \beta}^{\delta}=\frac{1}{2}g^{\delta\eta}(X_\alpha g_{\eta\beta}+X_{\beta}g_{\alpha\eta}-X_{\eta}g_{\alpha\beta}).
\end{align*}
It follows easily that 
\begin{align}\label{oncurve}
&g_{jk}|_{\Gamma}=\delta_{jk} \qquad \partial_i g_{jk}|_{\Gamma}=0\\ \nonumber
&g_{1k}|_{\Gamma}=0 \qquad g^{11}|_{\Gamma}=1.
\end{align}
For example, see the computations done in \cite{grey}. The above equations imply that we can think of the metric as being almost Euclidean in a neighborhood of the curve. In order to compute the Gaussian beam solutions on the manifold, we start with the following Lemmas which are essentially Lemmas 7.2, 7.3, and 7.4 in \cite{ks}. 

\begin{lem}
Let $(\mathcal{S}_0,g_0)$ be a Riemannian manifold without boundary, and let $\gamma:(a,b)\rightarrow \mathcal{S}_0$ be a unit speed geodesic segment with no loops. There are only finitely many points, $r\in (a,b)$ at which the geodesic $\gamma(r)$ intersects itself. 
\end{lem}
\begin{proof}
Because we assume that $\gamma(r)$ has no loops follows that $(\gamma(r),\dot{\gamma}(r))=(\gamma(r'),\dot{\gamma}(r'))$ implies that $r=r'$. The geodesic $\gamma$ may only intersect itself transversally, because by symmetry  $(\gamma(r),\dot{\gamma}(r))=(\gamma(r'),-\dot{\gamma}(r'))$ implies that $r=r'$. If the interval over which $\gamma$ is injective, say $\tilde{r}$ is smaller than diam$_{g_0}(\mathcal{S}_0)$, then any two geodesic segments which have length less than $\tilde{r}$ can intersect transversally at most one point. Partitioning the interval $(a,b)$ into disjoint intervals $\{I_l\}_{l=0}^L$, we then have an injective map
\begin{align*}
\{(r,r')\in (a,b)\times (a,b); r<r', \, \mathrm{and}\,\gamma(r)=\gamma(r')\} \rightarrow \\
\{(l,k)\in \{0, . . ,L\}\times \{0, . .,L\}; r\in I_l, r'\in I_j\}
\end{align*}
As a result, $\gamma$ can only intersect itself at finitely many points. 
\end{proof}

\begin{lem}
Let $F$ be a $C^1$ map from a neighborhood of $(a,b)\times \{0\}\in \mathbb{R}^{n}$ into a smooth manifold such that $F$ restricted to $(a,b)\times \{0\}$ is injective and also $DF(r,0)$ is invertible whenever $r\in (a,b)$. If we have that $[a_0,b_0]$ is a closed subinterval of $(a,b)$ then the map $F$ is a $C^1$ diffeomorphism in some neighborhood of $[a_0,b_0]\times \{0\}$ in $\mathbb{R}^n$. 
\end{lem}
\begin{proof}
For any $r\in [a_0,b_0]$ the inverse function theorem gives that there exists a $\epsilon_{r}>0$ such that $F$ restricted to $(r-\epsilon_r, r+\epsilon_r)\times B_{\epsilon_r}(0)$ is a $C^1$ diffeomorphism. Because the interval $[a_0,b_0]$ is compact, we can cover it with finitely many intervals $I_l$ of this form. In other words,
\begin{align*}
[a_0,b_0]\subset \bigcup_{l=0}^L(r_l-\epsilon_{r_l}, r_l-\epsilon_{r_l})
\end{align*}
and $F$ restricted to each of the intervals is bijective. With out loss of generality, we can rescale so that the intervals $\overline{I}_l\cap \overline{I}_k=\emptyset$ unless $|l-k|\leq 1$. Because $\gamma(r)=F(r,0)$ is injective, it follows that $\gamma(\overline{I}_l)\cap \gamma(\overline{I}_k)=\emptyset$, unless also $|l-k|\leq 1$. Let 
\begin{align*}
\delta=\inf \{d_{g_0}(\gamma(\overline{I}_i,\overline{I}_k)); |l-k|\geq 2\}>0
\end{align*}
Now also let $\mathcal{U}_l=I_l\times B_{\epsilon}(0)$ where $\epsilon<\min \{\epsilon_0,. . . ,\epsilon_L\}$ chosen sufficiently small so that $F(\mathcal{U}_l)\subseteqq \{q; d_{g_0}(q,\gamma(\overline{I}_l))<\delta\}$. 
 
We also let $F(\mathcal{U}_l)=O_l$. It follows that $O_l\cap O_j=\emptyset$ unless $|j-l|\leq 1$. Finally we let 
\begin{align*}
\mathcal{U}=\bigcup\limits_{l=0}^L\mathcal{U}_l, 
\end{align*}
and we note that $F$ restricted to $\mathcal{U}$ is a diffeomorphism as in \cite{ks}. 
\end{proof}

\begin{lem}\label{four}
Let $(\mathcal{S}_0,g_0)$ be a Riemannian manifold without boundary, and let $\gamma: (a,b)\rightarrow \mathcal{S}_0$ be a unit speed geodesic segment with initial data in $\mathcal{SM}^+$. Given a closed subinterval $[a_0,b_0]$ of $(a,b)$ such that $\gamma|_{[a_0,b_0]}$ self intersects at only finitely many points, $r_j$, with $a_0<r_1<. . . <r_n=b_0$. There is an open cover $\{O_l,\rho_l\}$ of $\gamma([a_0,b_0])$ of coordinate charts with the following properties
\begin{enumerate}
\item $\rho_l(O_l)=\mathcal{U}_l$
\item $\rho_l(\gamma(r))=(r,0) \qquad \forall r\in (r_l-\epsilon,r_l+\epsilon)$ 
\item $r_l$ belongs to $I_l$ and $\overline{I_l}\cap \overline{I_j}=\emptyset$ unless $|l-j|\leq 1$. 
\item $\rho_l=\rho_k$ on $\rho^{-1}_{l}(O_l\cap O_k)$.  
\end{enumerate}
\end{lem}
 
We can think of Fermi coordinates as a generalization of boundary normal coordinates. The geodesic tubes we consider as a generalization of the sphere- a simple manifold $\mathcal{M}$, which is an important but special case of the tubes. Much of the work done here will mimic the work done by \cite{DDSF} and the earlier work of \cite{sjo}, but instead we will be using the tubes because we have removed the assumption the manifold is simple.  In local coordinates, from (\ref{oncurve}) the nul-bicharacteristic equations (\ref{bich}) simplify so that
\begin{align}\label{bichlocal}
\frac{dx}{dt}=\frac{\xi}{|\xi|}  \qquad \frac{d\xi}{dt}=0
\end{align}
The solution to these equations is easily seen to be $(x(t),\xi(t))=((t,0)+x_0,(\lambda,0))$, where we recall that in local coordinates $\omega_0$ corresponds to the vector $(1,0)$, by choice of the basis for the tangent space. This simple computation shows the arclength coordinate $r$ on $\mathcal{M}$ is identified with the time $t$ on the boundary cylinder. In other words, we have that 
\begin{align*}
\rho_{l}(\gamma(t))=(t,0) 
\end{align*}
so that by Lemma \ref{four}, because $F$ is injective along the length of the curve, we have $\rho_l(\gamma(t))=x(t)$ in each $\mathcal{U}_l, l=0, . . ,L$. 

 We need to be able to compute the approximate solution in local coordinates which we do by introducing cutoff functions. We want a Gaussian beam defined in each coordinate chart which is localized there. Given a time interval $[0,T]$, we can cover it with finitely many intervals of length $2\epsilon$. Rescaling if necessary we can assume the corresponding local coordinate charts have diameter equal $2\epsilon$ as well. As before, we let 
\begin{align*}
\bigcup_{l=0}^L (t_l-\epsilon,t_l+\epsilon)\times O_l
\end{align*}
cover the graph $\{(t,x(t)): 0\leq t \leq T\}$.  

We now consider disjoint sets $(t_l,t_{l+1})\times \tilde{O}_l$ whose union is 
\begin{align*}
\{(t,x): |t-r-t_0|+ d_g(x,x(r))<2\epsilon_1^{\frac{1}{2n\alpha}}, 0\leq r\leq T\}
\end{align*}
but with $|t_l-t_{l+1}|, \mathrm{diam}_g(\tilde{O}_l)\leq 2\epsilon$ 
 From Lemma \ref{four}, we know that the neighborhoods $(t_l-\epsilon,t_l+\epsilon_l)\times \mathcal{U}_l$ can be identified with $(t-\epsilon_l,t+\epsilon_l)\times O_l$ which cover the graph $\{(t,\gamma(t)): 0\leq t \leq T\}$ by the use of Fermi coordinates.  We write
\begin{align*}
U_{\lambda}(t,x)\chi_{\epsilon_1}(t,x)|_{\tilde{O}_l}=U_{l,\lambda}(t,x)=a_l(t,x)\exp(i\psi_l(t,x))\chi_l(t,x)
\end{align*}
We will need to be able to compute the integral
\begin{align*}
\int\limits_0^T\int\limits_{\mathcal{S}_0}q(t,x)U_{\lambda}(t,x)\overline{W}_{\lambda}(t,x)\chi_{\epsilon_1}(t,x) \,d_{g_0}V \,dt.
\end{align*}
If we transfer everything to local coordinates via the map $F^{-1}$, and consider the restriction to each $O_l$, then we will be able to make sense of the integral using the previously defined sets. Now we make Lemma \ref{step2} more explicit
\begin{lem}\label{step2}
There exists a constants $C_1,C_2$ independent of $\lambda$, such that 
\begin{align*}
\sabs{\int\limits_{t_{l-1}}^{t_{l+1}}\int\limits_{\tilde{O}_l}q(t,x)\sparen{\frac{\lambda}{\pi}}^{\frac{n}{2}}|a_{l}(t,x)|^2\exp(-2\lambda \Im(\psi_l(t,x)))\chi_l(t,x)\,dV_{g_0}\,dt-\int\limits_{t_{l-1}}^{t_{l+1}}q(t,x(t))\,dt}\leq &\\ \frac{C_1\lambda^{\sigma}\epsilon_1^{-\frac{1}{2\alpha}}}{\sqrt{\lambda}}+C_2\mathrm{erfc}(-\lambda^{2\sigma})
\end{align*}
\end{lem}
Making the change of variables to the local coordinates, we see it suffices to prove Lemma \ref{H} to finish Lemma \ref{step2}.  The proof of Lemma \ref{step2} reduces to showing that for $h(t,x)\in C^1((0,T')\times O)$ with $O$ an open subset of $\mathbb{R}^n$, $T'<\infty$ and sufficiently large $\lambda$ we have 
\begin{align*}
\sabs{\int\limits_{0}^{T'}\int\limits_{O}\sparen{\frac{\lambda}{\pi}}^{\frac{n}{2}}h(t,x)\exp(-2\lambda |x-x(t)|^2)\,dx\,dt-\int\limits_0^{T'}h(t,x(t))\,dt}\leq &\\ \frac{C_1\lambda^{\sigma}\epsilon_1^{-\frac{1}{2\alpha}}}{\sqrt{\lambda}}+C_2\mathrm{erfc}(-\lambda^{2\sigma})
\end{align*}
where $x(t)=(t,0)$ using the notation above. This is proved in Lemma \ref{H}. The identification of $r$ and $t$ suggests that the time dependent X-ray transform may not always be enough to determine time dependent potentials. We leave it as an open question:

\textbf{What information can be gained about the potentials from the time dependent X-ray transform?}
\renewcommand\refname{\large References}
\bibliography{aldenbib}
\end{document}